\documentclass[reqno]{amsart}
\usepackage{amssymb, bm, amsmath, latexsym, mathabx, dsfont,tikz, float, mathrsfs}
\usepackage{makecell, multirow, longtable}
\usepackage{enumitem}

\renewcommand{\baselinestretch}{\baselinestretch}
\renewcommand{\baselinestretch}{1.1}
\numberwithin{equation}{section}

\newtheorem{thm}{Theorem}[section]
\newtheorem{lem}[thm]{Lemma}
\newtheorem{cor}[thm]{Corollary}
\newtheorem{prop}[thm]{Proposition}

\theoremstyle{definition}

\theoremstyle{remark}

\numberwithin{equation}{section}

\newcommand{\asu}[1][\mathcal{A}_{d,m}]{\mathrm{ASU}(#1)}
\newcommand{\su}[1][\mathcal{A}_{d,m}]{\mathrm{SU}(#1)}
\newcommand{\mca}{{\mathcal{A}}}

\newcommand{\ra}{{\, \rightarrow \,}}

\newcommand{\gen}{\text{gen}}

\newcommand{\pgen}{\text{gen}^+}
\newcommand{\pspn}{\text{spn}^+}
\newcommand{\pcls}{\text{cls}^+}
\newcommand{\ord}{\text{ord}}

\newcommand{\z}{{\mathbb Z}}

\newcommand{\q}{{\mathbb Q}}

\newcommand{\n}{{\mathbb N}}

\newcommand{\Mod}[1]{\ (\mathrm{mod}\ #1)}

\newcommand{\legendre}[2]{\left( \frac{#1}{#2} \right)}

\newcommand{\spnnorm}[1]{\theta\left(O^+({#1})\right)}

\newenvironment{newenum}
{\begin{enumerate}[label={\rm(\arabic*)}]}
	{\end{enumerate}}

\begin{document}


\author{Daejun Kim}
\address{Department of Mathematics Education, Korea University, Seoul 02841, Republic of Korea}
\email{daejunkim@korea.ac.kr}
\thanks{This research was supported by the College of Education, Korea University Grant in 2024.}



\thanks{}


\title[]{Sums of squares of integers from residue classes}

\begin{abstract} 
    A subset $\mathcal{A}\subseteq\mathbb{Z}$ is called $s$-almost square universal if every sufficiently large positive integer can be written as a sum of at most $s$ squares of integers from $\mathcal{A}$. In this article, we study the minimal number $\mathrm{ASU}(\mathcal{A}_{d,m})$ with this property, where $\mca_{d,m}$ denotes the residue class of $d$ modulo $m$, with $m\in\n$ and $d\in\mathbb{Z}$. We further prove that $\mathcal{A}_{d,m}$ is $s$-square universal for some $s\in\mathbb{N}$ if and only if $d \equiv \pm 1 \Mod{m}$, and determine the minimal such number $\mathrm{SU}(\mathcal{A}_{d,m})$ in these cases.
\end{abstract}

\maketitle

\section{Introduction}

Waring’s problem asks whether, for a given integer $k\ge2$, there exists a number $r$ such that every positive integer is a sum of at most $r$ $k$-th powers of positive integers. In 1770, the same year that Edward Waring posed his conjecture, Lagrange proved the Four Square Theorem, which states that every positive integer is a sum of at most four squares of integers. This theorem, though proved independently, resolved Waring’s problem in the special case $k=2$. In general, the existence of such a number was proved by Hilbert \cite{Hil} in 1909, and the determination of the smallest such number $g(k)$ has been extensively studied.

Building upon this, a closely related and indeed well-studied refinement considers the smallest number $G(k)$ such that every sufficiently large integer can be represented as a sum of at most $G(k)$ $k$-th powers of positive integers. For $k=2$, Legendre's Three Square Theorem implies that no positive integer of the form $4^c(8d+7)$ for integers $c,d$, can be expressed as a sum of three squares of integers. Thus, $g(2)=G(2)=4$. For more on these invariants $g(k)$ and $G(k)$, we refer the reader to the survey notes \cite{E} and \cite{VW}, and the references therein.

In what follows, we focus on the case $k=2$, that is, sums of squares. This naturally leads to the study of positive definite integral quadratic forms and their universality properties. A classical example is Ramanujan’s classification \cite{Ram} of 55 diagonal quaternary quadratic forms that are {\it universal}---that is, quadratic forms representing all positive integers. Dickson \cite{Di1} refined this list by removing the form $x^2+2y^2+5z^2+5w^2$, which fails to represent $15$. More recently, Conway and Schneeberger proved the celebrated 15-Theorem, which gives a complete criterion for universality: a positive definite integral quadratic form is universal if and only if it represents the integers $1,2,3,5,6,7,10,14,$ and $15$. Bhargava \cite{B} later gave a simplified proof and, together with Hanke \cite{BH}, extended the result to the 290-Theorem for integer-valued quadratic forms.

Beyond classical universality results, there has been considerable interest in representations by sums of squares with restricted variables, such as restrictions to almost primes. Br\"udern and Fouvry \cite{BrudernFouvry} proved that every sufficiently large integer
$n\equiv4\Mod{24}$ can be expressed as a sum of four squares of integers, each having at most 34 prime factors, and Blomer and Br\"udern \cite{BlomerBrudern} obtained a corresponding result for three squares.

Motivated by these developments, a natural generalization is to consider sums of squares drawn from specific subsets $\mca\subseteq\z$, leading to a Waring-type problem. We say $\mca$ is {\it $s$-square universal} (resp. {\it $s$-almost square universal}) if every positive integer (resp. every sufficiently large integer) can be written as a sum of at most $s$ squares from $\mca^2$. The minimal such $s$, if it exists, is denoted by $\mathrm{SU}(\mca)$ (resp. $\mathrm{ASU}(\mca)$). 

For example, in \cite{KO}, this problem was studied for $\mca=\z\setminus p\z$ with each prime $p$, where it is shown that 
\[
\su[\z\setminus 2\z]=10, \ \su[\z\setminus 3\z]=6, \ \su[\z\setminus 5\z]=5, \text { and } \su[\z\setminus p\z]=4 \text{ for any prime } p\ge 7.
\]
Although these stated results concern square universality, their proofs in fact establish the corresponding almost square universality results as well:
\[
\asu[\z\setminus 2\z]=10, \  \asu[\z\setminus 3\z]=6, \ \asu[\z\setminus p\z]=4 \text{ for any prime }p\ge5. 
\]
More recently, the case $\mca=S_\rho:=\z\setminus\{\pm\rho\}$ with $\rho\in\n$ was studied in \cite{CJKKOY}. Regarding almost square universality, it is shown that $\asu[S_\rho]=4$ for all $\rho\in\n$. For square universality, one may observe that $\su[S_1]$ does not exist, since $S_1$ excludes $\{\pm1\}$. For other cases, it was proved that
\[
    \su[S_\rho]=\begin{cases}
        8 & \text{if } \rho=2,\\
        6 & \text{if } \rho=3,\\
        5 & \text{if } \rho=5,\ 2^{m+1}, \ 3\cdot 2^{m} \text{ for some }m\in\n,\\
        4 & \text{otherwise}.\\
    \end{cases}
\]

In this article, we study the above concepts for the set $\mca=\mca_{d,m}$, the residue class of $d$ modulo $m$, where $m\in\n$ and $d\in\z$. Under the natural assumption that $m$ and $d$ are relatively prime---since otherwise $\asu$ cannot exist---we obtain the following main result characterizing almost square universality.

\begin{thm}\label{thm:asu-main}
For $m\in\n$ and $d\in\z$ with $(m,d)=1$, we have
    \[
        \asu=\begin{cases}
        	6 & \text{if } m=3,\\
            m+4 & \text{if } m\equiv 1\Mod{2} \text{ with } m\ge 5,\\
            4m+2 & \text{if } m\equiv 2\Mod{4},\\
            2m+2 & \text{if } m\equiv 0\Mod{4} \text{ and } m\not\equiv0\Mod{3},\\
            2m+3 & \text{if } m\equiv 0\Mod{12}.
        \end{cases} 
    \]
\end{thm}

Moreover, we characterize when $\mca_{d,m}$ is $s$-square universal for some $s\in\n$, and determine the exact value of $\su$ in those cases. 

\begin{thm}\label{thm:su-main}
The set $\mca_{d,m}$ is $s$-square universal for some $s\in\n$ if and only if $d\equiv\pm 1\Mod{m}$. Furthermore, for such integers $d$, we have
    \[
        \su[\mca_{d, m}]=\begin{cases}
            10 & \text{if } m=2,4,\\
            6 & \text{if } m=3,\\
            16 & \text{if } m=5,\\
            26 & \text{if } m=6,\\
            m^2-2m & \text{if } m\ge7.\\
        \end{cases}
    \]    
\end{thm}

This article is organized as follows. In Section \ref{sec:prelim}, we review preliminary results for quadratic spaces, lattices, and lattice cosets. Section \ref{sec:asu} is devoted to proving almost square universality of $\mca_{d,m}$ and Theorem \ref{thm:asu-main}. Finally, Section \ref{sec:su} addresses the proof of square universality, establishing Theorem \ref{thm:su-main}.

\section{Preliminaries}\label{sec:prelim}

In this section, we introduce notation and terminology of quadratic spaces, lattices, and lattice cosets, and their theta series.

\subsection{Quadratic spaces and lattices} 
A {\it quadratic space} over $\q$ is a triple $(V,B,Q)$ of a $\q$-vector space $V$ with a quadratic map $Q:V\ra \q$ and the associated bilinear form $B:V\times V \ra \q$ defined by $B(x,y)=\frac{1}{2}(Q(x+y)-Q(x)-Q(y))$. We say that $V$ is {\it positive definite} if $Q(x)>0$ for all nonzero elements $x\in V$. Throughout this article, we assume that all quadratic spaces $V$ are positive definite.

A {\it quadratic $\z$-lattice} on $V$ is a free $\z$-module $L=\z e_1+\cdots +\z e_r$ of rank $r$ such that $\q L=V$. The $r\times r$ symmetric matrix $M_L:=B(e_i,e_j)$ is called the {\it Gram matrix} of $L$ with respect to the basis $e_1,\ldots,e_r$, and we write $L\cong M_L$. If the Gram matrix $M_L$ is diagonal, we simply write $L\cong \langle Q(e_1),\ldots, Q(e_r)\rangle$. The quadratic form associated to $L$ is defined by $Q_L(x_1,\ldots, x_r):=\bm{x}^t M_L \bm{x}$ where $\bm{x}=(x_1,\ldots,x_r)^t$. The {\em discriminant} $d_L$ of $L$ is the determinant of the Gram matrix $M_L$, and the {\em level} $N_L$ is defined to be the smallest positive integer $N$ such that $NM_L^{-1}$ has coefficients in $\z$. We define the {\em localization} of $V$ and $L$ at any prime spot $p$ by $V_p:= V\otimes_\q \q_p$ and $L_p=L\otimes_\z \z_p$.

For two $\z$-lattices $\ell$ and $L$, we say that $\ell$ is represented by $L$ if there is a linear map $\sigma:\ell \ra L$ such that $B(\sigma x,\sigma y) = B(x,y)$ for all $x,y\in \ell$. Such a linear map is called a representation from $\ell$ to $L$, and we write $\ell\ra L$ if there exists one. If $\ell\ra L$ and $L\ra \ell$, then we say that they are isometric, and write $\ell\cong L$. 

These notations extend naturally to $\z_p$-lattices $\ell_p$ and $L_p$. We say that $\ell$ is {\it locally represented} by $L$ if $\ell_p\ra L_p$ for all primes $p$. The {\it genus} $\gen(L)$ of $L$ is defined by
\[
    \gen(L):=\{K \text{ on } \q L \mid K_p\cong L_p \text{ for all primes }p\}.
\]
The isometry induces an equivalence relation on $\gen(L)$, and the number of equivalence classes is called the {\it class number} of $L$.

It is well-known that if $\ell$ is locally represented by $L$, then there is a $\z$-lattice $L'\in\gen(L)$ such that $\ell\ra L'$ (see \cite[102:5]{OMBook}). Consequently, if $L$ is of class number one, then every locally represented $\ell$ is also globally represented by $L$. 

Utilizing this fact together with a detailed study of the representation of binary $\z$-lattice $\ell \cong \left(\begin{smallmatrix}4&b\\b&a\end{smallmatrix}\right)$ by the quaternary $\z$-lattice $L\cong\langle 1,1,1,1\rangle$, a generalization of Cauchy’s lemma was established in \cite{Kim}. This lemma plays a crucial role in our arguments. For earlier versions of related results, see \cite{N1}.

\begin{lem}\label{lem:GenCauchy}
	For $a,b\in\z$, the system of equations
	\begin{equation}\label{eqn:Cauchy}
	    \begin{cases}
			x_1^2+x_2^2+x_3^2+x_4^2=a\\
			x_1+x_2+x_3+x_4=b
		\end{cases}
	\end{equation}
	has a solution $x_1,x_2,x_3,x_4\in\z$ if and only if
	\begin{equation}\label{eqn:CauchyCond}
	    a\equiv b\Mod{2}, \quad 4a-b^2\ge0 \quad \text{and} \quad 4a-b^2\neq 4^\alpha(8\beta+7) \text{ for any $\alpha,\beta\in\z$}.
	\end{equation}
\end{lem}
\begin{proof}
    Since the case when $4a-b^2>0$ has been proved in \cite[Lemma 4.2]{Kim}, let us assume that $4a-b^2=0$. First, $b$ should be even since $4a=b^2$. Hence $a$ is also even as $a\equiv b\Mod{2}$. Writing $a=2a'$ and $b=2b'$, we have $2a'=(b')^2$, so $b'=2b''$ for some $b''\in\z$ and $a=4(b'')^2$. Thus, $x_i=b''$ is a solution to \eqref{eqn:Cauchy}.
\end{proof}

\begin{lem}\label{lem:GenCauchy5}
	For $a,b\in\z$, the system of equations
	\begin{equation}\label{eqn:Cauchy5}
		\begin{cases}
			x_1^2+x_2^2+x_3^2+x_4^2+x_5^2=a\\
			x_1+x_2+x_3+x_4+x_5=b
		\end{cases}
	\end{equation}
	has a solution $x_1,x_2,x_3,x_4,x_5\in\z$ if and only if $a\equiv b\Mod{2}$ and $5a-b^2\ge0$.
\end{lem}
\begin{proof}
	The ``only if" part follows from the fact that $x_i^2\equiv x_i\Mod{2}$ and the Cauchy--Schwartz inequality. Hence it suffices to prove the ``if" part. If $5a=b^2$, then $b=5b'$ for some $b'\in\z$, and hence $a=5b'^2$. In this case, setting $x_i=b'$ for all $i$ gives a solution. We therefore assume that $5a-b^2>0$. 
	
	Let $\ell=\z v_1+\z v_2\cong \left(\begin{smallmatrix}5&b\\b&a\end{smallmatrix}\right)$ and $L=\z e_1+ \cdots + \z e_5 \cong \langle 1,1,1,1,1 \rangle$ be $\z$-lattices. Then $\ell$ is locally represented by $L$ since their ranks differ by $3$ and $L_p$ is unimodular for every prime $p$. Thus $\ell$ is globally represented by $L$ as the class number of $L$ is one. Hence there exists a representation $\sigma:\ell \ra L$. 
	
	By changing the sign of the basis vectors $e_i$ and reordering them if necessary, we may assume that 
	\[
		\sigma(v_1)=e_1+e_2+e_3+e_4+e_5 \quad \text{ or } \quad \sigma(v_1)=2e_1+e_5.
	\]
	Write $\sigma(v_2)=\sum_{i=1}^5 y_i e_i$ for some $y_i\in\z$. Then $\sum_{i=1}^5 y_i^2=B(\sigma(v_2),\sigma(v_2))=B(v_2,v_2)=a$. In the former case, we have $\sum_{i=1}^5 y_i=B(\sigma(v_1),\sigma(v_2))=B(v_1,v_2)=b$, and hence $x_i=y_i$ gives a solution. 
	
	In the latter case, consider the sublattice $K=\z w_1+\z w_2\subseteq L$, where $w_1=2e_1$ and $w_2=y_1e_1+y_2e_2+y_3e_3+y_4e_4$. Since $K\cong \left(\begin{smallmatrix}4&b-y_5\\b-y_5&a-y_5^2\end{smallmatrix}\right)$ is a sublattice of $L$, it is positive semi-definite. Hence its discriminant $d_K=4(a-y_5^2)-(b-y_5)^2\ge0$. Since $K\subseteq \z e_1+\cdots+\z e_4\cong \langle 1,1,1,1\rangle$ in addition, $d_K$ is not of the form $4^\alpha(8\beta+7)$ for any $\alpha,\beta\in\z$. Finally, as $a-y_5^2\equiv b-y_5\Mod{2}$, Lemma \ref{lem:GenCauchy} yields $x_1,x_2,x_3,x_4\in\z$ satisfying
	\[
		\begin{cases}
			x_1^2+x_2^2+x_3^2+x_4^2=a-y_5^2\\
			x_1+x_2+x_3+x_4=b-y_5.
		\end{cases}
	\]
	Setting $x_5=y_5$ then gives a solution. This completes the proof of the lemma.
\end{proof}

\subsection{Quadratic lattice cosets} By a {\it lattice coset} or simply a {\it coset}, we mean a coset $L+\nu$ in $V/L$, where $\nu \in V$. The smallest integer $a$ such that $a\nu\in L$ is called the {\it conductor} of $L+\nu$. For example, let
\[
    L=\z e_1 + \z e_2 + \z e_3\cong \langle m^2,m^2,m^2\rangle \quad \text{and} \quad \nu=\frac{1}{m}(e_1+e_2+e_3),
\]
where $m\in\n$. Then the conductor of $L+d\nu$ is $m$ for any $d\in\z$ with $(m,d)=1$. Moreover, note that any vector $y\in L+d\nu$ is written as $y=x+d\nu$ with $x=x_1e_1+x_2e_2+x_3e_3\in L$ and
\[
    Q(y)=Q(x+d\nu)=Q\left(\sum_{i=1}^3 \left(x_i+\frac{d}{m} \right) e_i \right) = \sum_{i=1}^3 m^2\left(x_i+\frac{d}{m} \right)^2 = \sum_{i=1}^3 \left(mx_i+ d \right)^2.
\]
Thus, writing $n$ as a sum of three squares in $\mca_{d,m}^2$ is equivalent to finding a vector $y$ in $L+d\nu$ with $Q(y)=n$. We note that $L+\nu=L$ if and only if $\nu\in L$, equivalently, if the conductor of $L+\nu$ is $1$.

\subsection{Algebraic structures of lattice cosets}
Our approach requires an understanding of algebraic structures of lattice cosets: proper classes, proper genera, and  proper spinor genera. To introduce these concepts, let us denote the orthogonal group $O(V)$ and the special orthogonal group $O^+(V)$ of $V$ by
\[
    O(V):=\{\sigma\in \mathrm{GL}(V) \mid B(\sigma x , \sigma y)=B(x,y) \text{ for all }x,y\in V \} \quad \text{and} \quad O^+(V):=O(V)\cap \mathrm{SL}(V).
\]
Let $\theta:O^+(V)\rightarrow \q^\times/(\q^\times)^2$ be the {\em spinor norm map} (cf. \cite[$\S 55$]{OMBook}) and denote its kernel by
\[
O'(V)=\{\sigma\in O^+(V) : \theta(\sigma)=1\}.
\]
Let $O_A^+(V)$  and $O_A'(V)$ be the {\em ad{\'e}lizations} of $O^+(V)$ and $O'(V)$, respectively. Define
$$
O^+(L+\nu)=\{\sigma\in O^+(V) : \sigma(L+\nu)=L+\nu\} \quad \text{and} \quad o^+(L+\nu)=|O^+(L+\nu)|.
$$
Similarly, the groups $O^+(L_p+\nu)$ for any prime $p$ and $O_A^+(L+\nu)$ are defined analogously.

Let $X=L+\nu$ be a coset. The {\it proper class} $\pcls(X)$ of $X$ is defined as the orbit of $X$ under the action of $O^+(V)$. By \cite[Lemma 4.2]{ChanOh13}, the adelic group $O_A^+(V)$ acts on $X$, and the orbit of $X$ under this action $O_A^+(V)$ is called the {\it proper genus} of $X$, denoted by $\pgen(X)$. The {\it proper spinor genus} $\pspn(X)$ of $X$ is defined to be the orbit of $X$ under the action of the group $O^+(V)O_A'(V)$. Clearly, we have
$$
\pcls(X)\subseteq \pspn(X) \subseteq \pgen(X).
$$
Moreover, \cite[Proposition 2.5]{Xu} shows that the number of proper spinor genera in $\pgen(L+\nu)$ is given by
\begin{equation}\label{eqn:numberofproperspinorgenera}
\left[I_\q : \q^\times \prod\limits_{p\in \Omega} \spnnorm{L_p+\nu}\right],
\end{equation}
where $I_\q$ is the id{\`e}le group and $\Omega$ is the set of all places of $\q$, including the infinite place $\infty$.

\subsection{Modular forms}

We briefly introduce modular forms of half-integral weight below. For further details, we refer the reader to \cite{OnoBook}. A function $f:\mathbb{H}\ra \mathbb{C}$ satisfies a {\it modularity of weight $\kappa\in\frac{1}{2}+\z$ on a congruence subgroup $\Gamma\subseteq \Gamma_0(4)$ with character $\chi$} if $f|_\kappa=\chi(d)f$ for every $\gamma=\left(\begin{smallmatrix} a&b\\c&d \end{smallmatrix}\right) \in \Gamma$. Here the slash operator is defined by
\[
    f|_{\kappa}\gamma (z)= \legendre{c}{d}\varepsilon_d^{2\kappa}(cz+d)^{-\kappa}f(\gamma z),
\]
where $\legendre{c}{d}$ denotes the Kronecker symbol and $\varepsilon_d=1$ if $d\equiv 1 \pmod{4}$, $\varepsilon_d=i$ if $ d\equiv 3 \pmod{4}$. We call $f(z)$ a {\it holomorphic modular form} if it satisfies the modularity condition, is holomorphic on $\mathbb{H}$, and grows at most polynomially in $y$ as $z=x+iy\rightarrow \q\cup \{i\infty\}$. If, in addition, $f(z)\rightarrow 0$ as $z\rightarrow \q\cup \{i\infty\}$, then we call $f$ a {\it cusp form}. If $\left(\begin{smallmatrix} 1&1\\0&1 \end{smallmatrix}\right) \in \Gamma$, then a modular form $f(z)$ on $\Gamma$ admits the following Fourier expansion:
\[
    f(z)=\sum_{n\ge0} a_f(n)q^n, \quad \text{where }q:=e^{2\pi iz}.
\]

The space of modular forms (resp. cusp forms) of weight $\kappa$, character $\chi$, and congruence subgroup $\Gamma$ is denoted by $M_{\kappa}(\Gamma,\chi)$ (resp. $S_{\kappa}(\Gamma,\chi)$).
The space of {\em Eisenstein series}, denoted by $E_{\kappa}(\Gamma,\chi)$, is defined as the orthogonal complement of $S_{\kappa}(\Gamma,\chi)$ in $M_{\kappa}(\Gamma,\chi)$ with respect to the Petersson inner product.

\subsection{Theta series for cosets}
Let $X=L+\nu$ be a coset on a quadratic space $V$ of rank $k$, and assume throughout that $B(X,X)\subseteq \z$. For a positive integer $n$, we define
\[
R_X(n):=\{x\in X : Q(x)=n\} \quad \text{and} \quad r_X(n):= |R_X(n)|,
\]
and the theta series $\Theta_{X}(z)$ of the coset $X$ is defined as
\[
\Theta_{X}(z):=\sum\limits_{x\in X} q^{Q(x)} = \sum\limits_{n=0}^\infty r_X(n) q^n.
\]
Using \cite[Proposition 2.1]{Shimura}, one may show that $\Theta_X$ is a modular form of weight $\frac{k}{2}$ when $X$ is of rank $k$.


The theta series $\Theta_{\pgen(X)}(z)$ of $\pgen(X)$ and its Fourier coefficients $r_{\pgen(X)}(n)$ are defined by
\begin{equation}\label{defn-thetaofpropergenus}
	\Theta_{\pgen(X)}(z)=\sum\limits_{n=0}^\infty r_{\pgen(X)}(n)q^n:=\left(\sum\limits_{Y\in\pgen(X)}\frac{1}{o^+(Y)}\right)^{-1}\cdot \left(\sum\limits_{Y\in\pgen(X)}\frac{\Theta_{Y}(z)}{o^+(Y)}\right).
\end{equation}
In the same manner, we define the theta series $\Theta_{\pspn(X)}(z)$ of $\pspn(X)$ and $r_{\pspn(X)}(n)$ by
\begin{equation}\label{defn-thetaofproperspinorgenus}
	\Theta_{\pspn(X)}(z)=\sum\limits_{n=0}^\infty r_{\pspn(X)}(n) q^n:= \left(\sum\limits_{Y\in\pspn(X)}\frac{1}{o^+(Y)}\right)^{-1}\cdot \left(\sum\limits_{Y\in\pspn(X)}\frac{\Theta_{Y}(z)}{o^+(Y)}\right).
\end{equation}
Here, the summations run over a complete set of representatives of proper classes in $\pgen(X)$ or $\pspn(X)$, respectively.

\section{Almost square universality of $\mca_{d,m}$}\label{sec:asu}

In this section, we discuss the almost square universality of $\mca_{d,m}$ and prove Theorem \ref{thm:asu-main}. We first investigate necessary conditions and establish effective almost square universality.
\begin{prop}\label{prop:almost-univ}
	Let $m,d\in\n$ with $d \le m$ and $(m,d)=1$. For each $n\in\n$, write $n=mt+r_0d^2$ with unique integers $t$ and $0\le r_0 < m$. Define $M=1$ if $m\equiv 1\Mod{2}$; $M=4$ if $m\equiv 2\Mod{4}$; and $M=2$ if $m\equiv 0\Mod{4}$. 
    \begin{newenum}
        \item If $n$ can be written as a sum of $r$ squares in $\mca_{d,m}^2$, then $t\equiv \frac{r-r_0}{m} \Mod{M}$.
        
        \item For $m\equiv 1\Mod{2}$, every $n\ge \frac{m^4}{5}+md^2$ is a sum of at most $m+4$ squares in $\mca_{d,m}^2$.
        
        \item For $m\equiv 0 \Mod{2}$, every $n\ge \frac{m^4}{M^2}+Mmd^2$ is a sum of at most $Mm+3$ squares for $\mca_{d,m}^2$.
    
    \item If $\mca_{d,m}$ is $s$-almost square universal, then
        \[
        s \ge 
        \begin{cases}
            m+4 & \text{if } m\equiv 1\Mod{2} \text{ with } m\ge5,\\
            Mm+2 & \text{if } m\equiv 0\Mod{2}.
        \end{cases}
        \]
    \end{newenum}
\end{prop}

\begin{proof}
    (1) We first investigate necessary conditions for $n$ to be expressed as a sum of squares in $\mca_{d,m}^2$. Suppose that $n$ is written as below and note that
    \begin{equation}\label{eqn:equiv1}
        n=\sum_{i=1}^r (mx_i+d)^2 \quad \iff \quad mt+r_0d^2 = m\left(m\sum_{i=1}^r x_i^2 + 2d\sum_{i=1}^rx_i \right) + rd^2. 
    \end{equation}
    Therefore, we necessarily have $r\equiv r_0 \Mod{m}$, and hence \eqref{eqn:equiv1} is equivalent to 
    \begin{equation}\label{eqn:equiv2}
        t-\frac{r-r_0}{m}\cdot d^2 = m\sum_{i=1}^r x_i^2 + 2d\sum_{i=1}^rx_i.
    \end{equation}
    Writing $\sum_{i=1}^r x_i^2=a$ and $\sum_{i=1}^r x_i=b$, we have
    \[
        a\equiv b\Mod{2} \quad \text{and} \quad t-\frac{r-r_0}{m}\cdot d^2=ma+2db.
    \]
    Noting that $(m,d)=1$ and that $d\equiv 1\Mod{2}$ when $m\equiv 0\Mod{2}$, it follows that 
    \[
        \{ma+2db \mid a,b\in\z, \ a\equiv b \Mod{2}\}=M\z.
    \]
    Therefore, we necessarily have $t\equiv \frac{r-r_0}{m}\cdot d^2\equiv \frac{r-r_0}{m} \Mod{M}$.\\[5pt]
    (2) We choose $r\in\n$ to be the unique integer satisfying $5\le r \le m+4$ and $r\equiv r_0\Mod{m}$, as required by part (1) of the proposition. For $n\ge N$, we claim that there exist integers $a,b\in\z$ satisfying
    \begin{equation}\label{eqn:CondRep5}
    	t-\frac{r-r_0}{m}\cdot d^2=ma+2db, \quad  a\equiv b\Mod{2} \quad \text{and} \quad 5a-b^2\ge0.
    \end{equation}
    Once we prove the claim, by Lemma \ref{lem:GenCauchy5}, there exist  $x_1,x_2,x_3,x_4,x_5\in\z$ such that $\sum_{i=1}^5 x_i^2=a$ and $\sum_{i=1}^5 x_i=b$. Setting  $x_6=\cdots=x_r=0$, equations \eqref{eqn:equiv1} and \eqref{eqn:equiv2} then show that $n=\sum_{i=1}^r (mx_i+d)^2$.
    
    To prove the claim, for simplicity, set
    \[
    	\tilde{t}:=t-\frac{r-r_0}{m}d^2.
    \]
    Note that if $(a,b)=(a_0,b_0)$ is an arbitrary integer solution to $\tilde{t}=ma+2db$, then all integer solutions $(a,b)$ to this equation are given by 
    \[
    	(a_k,b_k)=(a_0-2dk,b_0+mk), \quad k\in\z.
    \]
    Moreover,
    \begin{align}
    	5a-b^2=5\cdot \frac{\tilde{t}-2db}{m}-b^2\ge 0 &\iff mb^2+10db-5\tilde{t}\le 0 \nonumber \\ 
    	&\iff b \in \left[-\frac{5d}{m}-\frac{\sqrt{25d^2+5m\tilde{t}}}{m},-\frac{5d}{m}+\frac{\sqrt{25d^2+5m\tilde{t}}}{m}\right]. \label{eqn:b-interval5}
    \end{align}
    Since $n=mt+r_0d^2=m\tilde{t}+rd^2\ge \frac{1}{5}m^4+md^2$, we obtain
    \[
    5m\tilde{t}\ge m^4+5(m-r)d^2 \ge m^4-25d^2.
    \]
    This implies that the length of the interval in \eqref{eqn:b-interval} is at least $2m$. Therefore, there exist at least two consecutive integer solutions $(a,b)$ of $\tilde{t}=ma+2db$ satisfying $5a-b^2\ge0$, one of which necessarily satisfies $a\equiv b\Mod{2}$. Hence, one can find a pair $(a,b)=(a_k,b_k)$ that satisfies \eqref{eqn:CondRep5}, as claimed.\\[5pt]
    (3) For $n\ge N$, we show that there exist integers $a,b\in\z$ and an integer $4\le r\le Mm+3$ such that
    \begin{equation}\label{eqn:CondRep}
        t-\frac{r-r_0}{m}\cdot d^2=ma+2db, \ \  a\equiv b\Mod{2}, \ \ 4a-b^2\ge0 \ \text{ and } \ 4a-b^2\neq 4^\alpha(8\beta+7) \text{ ($\alpha,\beta\in\z$)}.
    \end{equation}
    Since $a$ and $b$ satisfy \eqref{eqn:CauchyCond}, Lemma \ref{lem:GenCauchy} implies that there exist $x_1,x_2,x_3,x_4\in\z$ such that $\sum_{i=1}^4 x_i^2=a$ and $\sum_{i=1}^4 x_i=b$. Setting  $x_5=\cdots=x_r=0$, equations \eqref{eqn:equiv1} and \eqref{eqn:equiv2} then show that $n=\sum_{i=1}^r (mx_i+d)^2$.
    
    We choose $r\in\n$ to be the unique integer with $4\le r \le Mm+3$ satisfying $t\equiv \frac{r-r_0}{m}\Mod{M}$, as required by part (1) of the proposition. To explain the choice of $a$ and $b$, as before, set
    \[
        \tilde{t}:=t-\frac{r-r_0}{m}d^2.
    \]
Note that $\tilde{t}\equiv 0\Mod{M}$ and $d\equiv1\Mod{2}$. Suppose that $(a,b)=(a_0,b_0)$ is an arbitrary integer solution of $\frac{\tilde{t}}{2}=\frac{m}{2}a+db$. Then all integer solutions $(a,b)$ are given by 
\[
(a_k,b_k)=\left(a_0-dk,b_0+\frac{m}{2}k\right) \text{ with } k\in\z.
\]
We claim that there exists some integer $0\le k \le \frac{8}{M}-1$ such that  $(a,b)=(a_k,b_k)$ satisfies the second and the last conditions in \eqref{eqn:CondRep}.

Assuming the claim, we now construct such a pair $(a,b)$. Since $a=\frac{\tilde{t}-2db}{m}$, we have 
\begin{align}
	4a-b^2=4\cdot \frac{\tilde{t}-2db}{m}-b^2\ge 0 &\iff mb^2+8db-4\tilde{t}\le 0 \nonumber \\ 
	&\iff b \in \left[-\frac{4d}{m}-\frac{2\sqrt{4d^2+m\tilde{t}}}{m},-\frac{4d}{m}+\frac{2\sqrt{4d^2+m\tilde{t}}}{m}\right]. \label{eqn:b-interval}
\end{align}
Since $n=m\tilde{t}+rd^2\ge \frac{m^4}{M^2}+Mmd^2$, we obtain
\[
m\tilde{t}\ge \frac{m^4}{M^2}+ (Mm-r)d^2 > \frac{m^4}{M^2}-4d^2.
\]
Hence, the length of the interval in \eqref{eqn:b-interval} is greater than $\frac{8}{M}\cdot \frac{m}{2}$. Therefore, there are at least $\frac{8}{M}$ consecutive integer solutions $(a,b)$ of $\frac{\tilde{t}}{2}=\frac{m}{2}a+db$ with $4a-b^2>0$, among which one can find a pair $(a,b)=(a_k,b_k)$ satisfying \eqref{eqn:CondRep} by the claim.

We now prove the claim. If $m\equiv2\Mod{4}$, then $\frac{\tilde{t}}{2}\equiv0\Mod{2}$. Hence $a_j\equiv b_j\Mod{2}$ for all $j$, and there exists some $i\in\{0,1\}$ such that $a_i\equiv b_i\equiv 1\Mod{2}$. For this $i$, we have $4a_i-b_i^2\equiv 3\Mod{8}$, and we may take $k=i$.

If $m\equiv0\Mod{4}$, then let $i\in\{0,1\}$ be such that $a_{i+2j}\equiv b_{i+2j}\Mod{2}$ for all $j\in\z$. Set $D_k=4a_k-b_k^2$. If $a_i\equiv b_i \equiv 1 \Mod{2}$, then $D_i\equiv 3\Mod{8}$, and we may take $k=i$. Otherwise, assume $a_i\equiv b_i\equiv 0\Mod{2}$, and write $a_i=2a_i'$ and $b_i=2b_i'$. Then $D_i\equiv0\Mod{4}$, and 
\begin{equation}\label{eqn:D-residue}
	\frac{D_i}{4}=2a_i'-b_i'^2\equiv \begin{cases}
		1\Mod{4} & \text{if } b_i'\equiv1\Mod{2} \text{ and } a_i'\equiv 1\Mod{2},\\
		3\Mod{8} & \text{if } b_i'\equiv1\Mod{2} \text{ and } a_i'\equiv 2\Mod{4},\\
		2\Mod{4} & \text{if } b_i'\equiv0\Mod{2} \text{ and } a_i'\equiv 1\Mod{2}.
	\end{cases}
\end{equation}
In these cases, we may again take $k=i$. The remaining cases are: 
\begin{enumerate}
	\item[(i)] $b_i'\equiv1\Mod{2}$ with $a_i'\equiv 0\Mod{4}$; and 
	\item[(ii)] $b_i'\equiv 0\Mod{2}$ with $a_i'\equiv 0\Mod{2}$. 
\end{enumerate}
	In case (i), since $D_i/4\equiv 7\Mod{8}$, we have
\[
D_{i+2}/4=D_i/4-\left[(2d+b_i'm)+\left(\frac{m}{2}\right)^2\right]\equiv 3-2\equiv1\Mod{4}.
\]
In case (ii), we have $a_{i+2}'=a_{i}'-d\equiv1\Mod{2}$ and $b_{i+2}'=b_{i}'+\frac{m}{2}\equiv 0\Mod{2}$., and hence $D_{i+2}/4=2a_{i+2}'-b_{i+2}'^2\equiv 2\Mod{4}$. Thus, in either case,  we may take $k=i+2$. This proves the claim and completes the proof of part (3).\\ [5pt]
(4) First, suppose that $m\equiv 0\Mod{2}$. Let $p$ be a prime such that $p\equiv 3\Mod{4}$ and $(p,mM)=1$. Consider a positive integer $n=mMt_0+2d^2$, where $t_0$ is a positive integer satisfying
\[
    mMt_0\equiv -2d^2 \Mod{p} \quad \text{and} \quad  mMt_0\not\equiv -2d^2 \Mod{p^2}.
\]
Since $\ord_p(n)=1$, the integer $n$ cannot be expressed as a sum of two squares of integers. On the other hand, by part (1), writing $n=mMt_0+2d^2$ as a sum of $r$ squares in $\mca_{d,m}^2$ requires that $r\equiv 2\Mod{Mm}$. Hence $r\ge mM+2$. Since there exist infinitely many such $n$, we conclude that $\asu\ge Mm+2$.

Now suppose that $m\equiv 1\Mod{2}$. We construct infinitely many positive integers $n_j=mt_j+4d^2$ for $j\ge 0$ that cannot be expressed as a sum of four squares in $\mca_{d,m}^2$. By part (1), expressing $n_j$ as a sum of elements of $\mca_{d,m}^2$ requires at least $m+4$ squares, which implies that $\asu\ge m+4$.

We first choose such an integer $n_0$ with $t_0\in 2\z$. We begin with the case when $0< d \le \frac{m-1}{2}$. Note that the three smallest squares in $\mca_{d,m}^2$ are $d^2$, $(-m+d)^2$, $(m+d)^2$, listed in increasing order. Consequently, the smallest two integers and the next several integers that can be expressed as sums of four squares in $\mca_{d,m}^2$ are 
\begin{center}
	$4d^2$, $3d^2+(-m+d)^2$, $2d^2+2(-m+d)^2$, $d^2+3(-m+d)^2$, $4(-m+d)^2$, and  $3d^2+(m+d)^2$.
\end{center}
Writing each of these integers in the form $mt+4d^2$, the corresponding values of $t$ are
\begin{center}
	$0$, $m-2d$, $2(m-2d)$, $3(m-2d)$, $4(m-2d)$, and  $m+2d$.
\end{center}
If $d<\frac{m-1}{2}$, then $m-2d\ge3$, and hence we may take $t_0=2$. If $d=\frac{m-1}{2}$, then $m+2d=2m-1\ge9$, so we may take $t_0=6$.

Next, consider the case $\frac{m+1}{2} \le d \le m-1$. Note that the three smallest squares in $\mca_{d,m}^2$ are $(-m+d)^2$, $d^2$, $(-2m+d)^2$, listed in increasing order. Hence, the smallest two and the next several integers that can be expressed as sums of four squares in $\mca_{d,m}^2$ are
\begin{center}
	$4(-m+d)^2$, $3(-m+d)^2+d^2$, $2(-m+d)^2+2d^2$, $(-m+d)^2+3d^2$, $4d^2$, and  $3(-m+d)^2+(-2m+d)^2$.
\end{center}
Writing each of these integers as $mt+4d^2$, the corresponding values of $t$ are
\begin{center}
	$4(m-2d)$, $3(m-2d)$, $2(m-2d)$, $m-2d$, $0$, and  $7m-10d$.
\end{center}
If $d>\frac{m+1}{2}$, then $m-2d\le -3$, and hence we may choose $t_0$ to be an even integer strictly between $4(m-2d)$ and $3(m-2d)$. If $d=\frac{m+1}{2}$, then $7m-10d=2m-5\ge5$, and we may take $t_0=2$.

For each $j\in\n$, we define
\[
	t_j:=\frac{4^{\ell j}(mt_0+4d^2)-4d^2}{m} \quad \text{and} \quad n_j:=mt_j+4d^2=4^{\ell j}(mt_0+4d^2),
\]
where $\ell:=\ord_m(2)=\ord_m\left(\frac{m+1}{2}\right)$. Since the numerator of $t_j$ is a multiple of $m$, we have $t_j\in\n$. 

We claim that $n_j$ cannot be written as a sum of four squares in $\mca_{d,m}^2$ for any $j$. Assume to the contrary that
\[
	n_j=\sum_{i=1}^{4} (mx_i+d)^2
\]
for some $x_i\in\z$. Since $n_j$ is a multiple of $8$ (as $t_0$ is even), it follows that $mx_i+d\equiv 0 \Mod{2}$ for all $i$. Noting that $\frac{mx_i+d}{2} \equiv d\cdot \frac{m+1}{2}\Mod{m}$, we have
\[
	4^{\ell j-1}(mt_0+4d^2)=\sum_{i=1}^{4} \left(my_i+d\cdot \frac{m+1}{2}\right)^2
\]
for some $y_i\in\z$. Iterating this argument $\ell-1$ more times, we obtain
\[
n_{j-1}=4^{\ell (j-1)}(mt_0+4d^2)=\sum_{i=1}^{4} \left(mz_i+d\cdot \left(\frac{m+1}{2}\right)^\ell\right)^2
\]
for some $z_i\in\z$. Since $\ord_{m}\left(\frac{m+1}{2}\right)=\ell$, we have 
\[
	n_{j-1}=\sum_{i=1}^{4} \left(mX_i+d\right)^2
\]
for some $X_i\in\z$. Repeating this procedure, we eventually obtain a representation of $n_0$ as a sum of four squares in $\mca_{d,m}^2$, contradicting the choice of $n_0$. This proves the claim and completes the proof.
\end{proof}

As a direct consequence of Proposition \ref{prop:almost-univ} and from the fact that $\asu[\mca_{d,3}]=\asu[\z\setminus 3\z]=6$ for any $d\in\z$ with $(d,3)=1$ by \cite{KO}, we obtain the following corollary concerning $\asu$.

\begin{cor}\label{cor:asu}
    Let $m\in\n$ and $d\in\z$ be such that $(m,d)=1$, and define $M=4$ if $m\equiv2\Mod{4}$ and $M=2$ if $m\equiv0\Mod{4}$. Then
    \[
        \asu=\begin{cases}
        	6 & \text{if } m=3,\\
            m+4 & \text{if }m\equiv 1\Mod{2},\\
            Mm+2 \text{ or } Mm+3 & \text{if }m\equiv 0\Mod{2}.
        \end{cases}
    \]
    Moreover, for $m\equiv0\Mod{2}$, we have $\asu=Mm+2$ if and only if every sufficiently large integer $n\equiv 3d^2\Mod{Mm}$ can be written as a sum of three squares in $\mca_{d,m}^2$.
\end{cor}

According to Corollary \ref{cor:asu},  determining $\asu$ reduces to the case when $m$ is even and to expressing every sufficiently large integer $n\equiv 3d^2\Mod{Mm}$ as a sum of three squares in $\mca_{d,m}^2$. We address this problem by studying representations of integers by ternary lattice cosets.

\begin{thm}\label{thm:num-pspn}
   Let $m\in 2\n$ and $d\in\z$ with $(m,d)=1$. Let $L=\z e_1+\z e_2+ \z e_3\cong \langle m^2,m^2,m^2\rangle$ be a ternary $\z$-lattice and let $\nu=\frac{1}{m}(e_1+e_2+e_3)\in\q L$. Consider the lattice coset $X^d=L+d\nu$. 
   \begin{enumerate}[label={\rm (\arabic*)}]
       \item If $m\equiv 2\Mod{4}$ or $m\equiv0\Mod{4}$ with $3\nmid m$, then $\pspn(X^d)=\pgen(X^d)$.
       \item If $m\equiv0\Mod{12}$, then there are two proper spinor genera in $\pgen(X^d)$.
   \end{enumerate}
\end{thm}
\begin{proof}
    Recall from \eqref{eqn:numberofproperspinorgenera} that the number of proper spinor genera in $\pgen(X^d)$ is given by
	\begin{equation}\label{eqn:num-pspns}
	\left[I_\q : \q^\times \prod\limits_{p\in \Omega} \spnnorm{X_p^d}\right].
	\end{equation}
    Hence, it remains to compute $\spnnorm{X_p^d}$ for all primes $p$. In fact, these can be obtained from the previously calculated local spinor norm groups for lattice cosets considered in \cite{BKK} and \cite{HK}. Therefore, instead of providing an explicit computation of $\spnnorm{X_p^d}$, we explain how they can be deduced from those results.
    
    First note that since $(m,d)=1$, \cite[Lemma 3.1 (1)]{KaneKim} implies that $O^+(L_p+d\nu)=O^+(L_p+\nu)$ for all primes $p$. Therefore, for each prime $p$, the spinor norm group $\spnnorm{X_p^d}$ is independent of $d$ as long as $(m,d)=1$.

    Assume $m\equiv 2\Mod{4}$ and $p=2$. In this case, we have $L_2=\z_2 e_1+\z_2 e_2+\z_2 e_3\cong\langle 4,4,4\rangle$ and 
    \[
        X_2^d=L_2+\nu=L_2+\frac{1}{m}(e_1+e_2+e_3)=L_2+\frac{1}{2}(e_1+e_2+e_3).
    \]
    By \cite[Proposition 2.3]{Kim}, it follows that $O^+(X_2^d)=O^+(L_2)$, and hence $\spnnorm{X_2^d}=\z_2(\q_2^\times)^2$ as described in the proof of \cite[Proposition 4.2]{BKK}.

    Otherwise, either $p$ is an odd prime or $p=2$ with $m\equiv 0\Mod{4}$. In this case, the spinor norm group $\spnnorm{X_p^d}$ has already been computed in \cite[Proposition 3.4]{HK} under different setting. To explain this, consider the lattice coset $K+\mu$ defined by
    \[
        K=K_{(m')}=\z f_1+\z f_2+\z f_3\cong \langle(m'-2)^2,(m'-2)^2,(m'-2)^2\rangle \quad \text{and} \quad \mu:=\frac{m'-4}{2(m'-2)}(f_1+f_2+f_3)
    \]
    as in \cite[(2.1) and (2.2)]{HK}. By setting $m'=m+2$, we observe that $K+\mu=L+\left(\frac{m-2}{2}\right)\nu$, and $(m,\frac{m-2}{2})=1$ since $m\equiv 0\Mod{4}$. Therefore, for all $d$ with $(m,d)=1$, we have $\spnnorm{X_p^d}=\spnnorm{K_p+\mu}$. Note also that
    \[
        \begin{cases}
            m'\equiv 2\Mod{4} \iff m \equiv 0\Mod{4},\\
            m'\equiv 2\Mod{3} \iff m \equiv 0\Mod{3},
        \end{cases}
    \]
    so the conditions for $m'$ in \cite[Proposition 3.4]{HK} correspond exactly to the condition on $m$ here. Hence,
    \[
        \theta(O^+(X^d_p))=\begin{cases}
            (\q_3^\times)^2 & \text{if } p=3 \text{ and } p\mid m,\\
            \z_p^\times(\q_p^\times)^2 & \text{otherwise}.
        \end{cases}
    \]        
    Furthermore, when $m\equiv 0\Mod{4}$, the index in \eqref{eqn:num-pspns}, which counts the number of proper spinor genera in $\pgen(X^d)$, can also be computed as in the proof of \cite[Proposition 3.4]{HK}. This completes the proof for the case when $m\equiv 0\Mod{4}$.
    
    Now we consider the case when $m\equiv 2\Mod{4}$. Since the computation of $\theta(O^+(X^d_p))$ for odd primes $p$ in \cite[Proposition 3.4]{HK} applies unchanged in this setting, the above arguments yield
	\[
		\spnnorm{X_p^d} =
		\begin{cases} 
			\q_2\setminus \{0\} & \text{if } p=2,\\
			(\q_3^\times)^2& \text{if } p=3 \text{ and } 3\mid m,\\
			\z_p^\times (\q_p^\times)^2 & \text{otherwise}.
		\end{cases}	
	\]
    The index \eqref{eqn:num-pspns} equals $1$ as computed in the proof of \cite[Theorem 4.3]{BKK}, completing the proof.
\end{proof}

Using the decomposition of theta series into three parts, we obtain the following theorem.

\begin{thm}\label{thm:asu-ternary}
    Let $m\in\n$ be an even integer such that $12\nmid m$, and let $d\in\z$ satisfy $1\le d\le m$ with $(m,d)=1$. Then every sufficiently large integer $n$ with $n\equiv 3d^2\Mod{m}$ can be represented in the form 
    \begin{equation}\label{eqn:three-squares-m,d}
        n=\sum_{i=1}^3 (mx_i+d)^2
    \end{equation}
    for some integers  $x_i\in\z$.
\end{thm}

\begin{proof}
Let $L=\z e_1+\z e_2+ \z e_3\cong \langle m^2,m^2,m^2\rangle$ and $\nu=\frac{1}{m}(e_1+e_2+e_3)\in\q L$. Consider $X^d=L+d\nu$. Recall that $n$ can be written as in \eqref{eqn:three-squares-m,d} if and only if $n$ is represented by $X^d$. Since $\pgen(X^d)=\pspn(X^d)$ by Theorem \ref{thm:num-pspn}, and recalling definitions \eqref{defn-thetaofpropergenus} and \eqref{defn-thetaofproperspinorgenus}, we have the decomposition
\[
\Theta_{X^d}= \Theta_{\pgen(X^d)}+\Theta_{X^d}-\Theta_{\pgen(X^d)}= \Theta_{\pgen(X^d)}+\Theta_{X^d}-\Theta_{\pspn(X^d)}.
\]
By \cite[Theorem 6.2]{KaneKim}, the difference $\Theta_{X^d}-\Theta_{\pspn(X^d)}$ lies in the subspace of cusp forms orthogonal to unary theta functions. Hence, by Duke's bound \cite{Duke}, the absolute value of the Fourier coefficients of $\Theta_{X^d}-\Theta_{\pspn(X^d)}$ satisfy the estimate $\ll_{m,\varepsilon}n^{3/7+\varepsilon}$. 
On the other hand, the Fourier coefficients $r_{\pgen(X^{d})}(n)$ of $\Theta_{\pgen(X^d)}$ satisfy the lower bound $r_{\pgen(X^d)}(n)\gg_{m,\varepsilon} n^{\frac{1}{2}-\varepsilon}$, since $n$ is locally  represented by $X^d$ and the order $\ord_p(n)=0$ for every prime divisor $p\mid m$. Comparing these asymptotic behaviors, we conclude the proof of theorem.
\end{proof}

\begin{thm}[{\cite[Theorem 5.2]{HK}}]\label{thm:exceptional}
   $\mathrm{(1)}$ If $\ell\equiv1 \Mod{12}$ is an odd prime, then 
   \[
    x^2+y^2+z^2=3\ell^2
   \]
   has no solution in $x,y,z\in\z$ with $x\equiv y\equiv z\equiv 5\Mod{12}$.

   $\mathrm{(2)}$ If $\ell\equiv 7 \Mod{12}$ is an odd prime, then 
   \[
    x^2+y^2+z^2=3\ell^2
   \]
   has no solution in $x,y,z\in\z$ with $x\equiv y\equiv z\equiv 1\Mod{12}$.
\end{thm}

With all the necessary results established, we are ready to complete the proof of Theorem \ref{thm:asu-main}.

\begin{proof}[Proof of Theorem \ref{thm:asu-main}]
    By Corollary \ref{cor:asu} and Theorem \ref{thm:asu-ternary}, we have 
    \[
        \asu=\begin{cases}
            m+4 & \text{if }m\equiv 1\Mod{2},\\
            4m+2 & \text{if }m\equiv 2\Mod{4},\\
            2m+2 & \text{if }m\equiv 0\Mod{4} \text{ and } m\not\equiv0\Mod{3}.
        \end{cases}
    \]
    Thus, it remains to show that $\asu=2m+3$ if $m\equiv 0\Mod{12}$. 
    
    Write $m=12m_0$ with $m_0\in\n$, and let $d_0\in\{1,5\}$ be such that $d\equiv \pm d_0 \Mod{12}$. Consider an odd prime $\ell$ satisfying
    \begin{equation}\label{eqn:ell-cond}
        \ell\equiv \begin{cases}
            7 \Mod{12} & \text{if } d\equiv 1 \text{ or } 11 \Mod{12}\\
            1 \Mod{12} & \text{if } d\equiv 5 \text{ or } 7 \Mod{12}.
        \end{cases}
    \end{equation}
   Assume that $n=mt+3d^2=3\ell^2$ and that $n$ is a sum of three squares in $\mca_{d,m}^2$. Then
    \begin{equation}
        3\ell^2=n=\sum_{i=1}^3 (mx_i+d)^2=(12m_0x_1+d)^2+(12m_0x_2+d)^2+(12m_0x_3+d)^2.
    \end{equation}
    for some $x_i\in\z$. By changing the sign of all $x_i$ if necessary, the equation $x^2+y^2+z^2=3\ell^2$ has a solution with $x\equiv y \equiv z\equiv d_0 \Mod{12}$, which contradicts Theorem \ref{thm:exceptional}. 
    
    Hence, by Corollary \ref{cor:asu}, it suffices to show that there are infinitely many odd primes $\ell$ satisfying \eqref{eqn:ell-cond} and $3\ell^2\equiv 3d^2\Mod{Mm}$, equivalently, $\ell^2\equiv d^2\Mod{4Mm_0}$, where $M=4$ if $m\equiv2\Mod{4}$ and $M=2$ if $m\equiv0\Mod{4}$. Write $Mm_0=2^a3^bm_1$ with $(m_1,6)=1$. Then by the Chinese Remainder Theorem and Hensel's Lemma, there exists $\ell$ such that $\ell\equiv d\Mod{m_1}$ and
    \[
    \begin{cases}
        \ell\equiv 1\Mod{3},\ \ell\equiv \pm d\Mod{3^b},\ \ell \equiv -1\Mod{4},\ \ell^2\equiv d^2\Mod{2^{a+2}} & \text{if }d_0=1,\\
        \ell\equiv 1\Mod{3},\ \ell\equiv \pm d\Mod{3^b},\ \ell \equiv 1\Mod{4},\ \ell^2\equiv d^2\Mod{2^{a+2}} & \text{if }d_0=5.
    \end{cases}
    \]
    Such $\ell$ satisfies $\ell^2\equiv d^2 \Mod{4Mm_0}$. By Dirichlet's theorem on arithmetic progression. This concludes the proof that $\asu[\mca_{d,m}]=2m+3$ when $m\equiv0\Mod{12}$.
\end{proof}

\section{Square universality of $\mca_{d,m}$}\label{sec:su}

In this section, we discuss the square universality of $\mca_{d,m}$ and prove Theorem \ref{thm:su-main}. We first investigate the necessary conditions.

\begin{lem}\label{lem:su-necessary}
Let $m\in\n$ and $d\in\z$ with $(m,d)=1$. If $\mca_{d,m}$ is square universal, then we have
\[
d\equiv\pm1\Mod{m} \quad \text{and} \quad \su[\mca_{d,m}]  \ge \max(m^2-2m,\asu)
\]
for these integers $d$.
\end{lem}
\begin{proof}
    To represent $1$ as a sum of squares in $\mca_{d,m}^2$, we must have either $1\in\mca_{d,m}$ or $-1\in\mca_{d,m}$. Hence, $d\equiv \pm1\Mod{m}$. Moreover, to represent $(m-1)^2-1=m^2-2m$, we need at least $m^2-2m$ squares in $\mca_{d,m}^2$, since $(m-1)^2$ is the second smallest square in $\mca_{d,m}^2$. Therefore, $\su[\mca_{d,m}] \ge m^2-2m$. Since by definition $\su\ge\asu$, the claim follows.
\end{proof}

\begin{lem}\label{lem:rep-small}
	Let $m\ge8$ be an integer, and let $d$ be an integer with $d\equiv\pm1\Mod{m}$. Then every positive integer $n\le 2m(m-1)^2$ can be written as a sum of at most $m^2-2m$ squares in $\mca_{d,m}^2$.
\end{lem}

\begin{proof}
Without loss of generality, we may assume $d=1$. Recall that the three smallest squares in $\mca_{1,m}^2$ are $1$, $(m-1)^2$, and $(m+1)^2$. Let $k,r$ be integers such that $0\le k \le 2m-1$ and $k\le r\le m^2-2m$. Setting $x_1=\cdots=x_k=-1$ and $x_{k+1}=\cdots=x_r=0$, we have
\[
    \sum_{i=1}^r (mx_i+1)^2 = k\cdot (m-1)^2+ (r-k).
\]
Hence, every integer $n$ such that
\[
k(m-1)^2\le n \le k(m-1)^2+(m^2-2m-k)=(k+1)(m-1)^2-(k+1)
\]
can be represented as a sum of squares in $\mca_{1,m}^2$. Therefore, all integers $n\le2m(m-1)^2$ can be represented in this manner, except those $n$ satisfying
\begin{equation}\label{eqn:n-left}
(k+1)(m-1)^2-k\le n\le (k+1)(m-1)^2-1 \quad \text{with }1\le k \le 2m-1.
\end{equation}

Now, taking $x_1=1$, $x_2=\cdots=x_k=-1$, and $x_{k+1}=\cdots=x_r=0$, we obtain
\[
    \sum_{i=1}^r (mx_i+1)^2 = (m+1)^2+(k-1)\cdot(m-1)^2+ (r-k)=(k+1)\cdot (m-1)^2-(m^2-6m+1)+(r-k).
\]
Since $k\le r \le m^2-2m$, every integer $n$ in the range
\[
(k+1)\cdot (m-1)^2-(m^2-6m+1)\le n \le (k+1)\cdot (m-1)^2+(4m-k-1)
\]
is representable as a sum of squares in $\mca_{1,m}^2$. Since $k\le 2m-1\le m^2-6m+1$ and $4m-k-1\ge 4m-(2m-1)-1\ge -1$, this range covers all integers $n$ in \eqref{eqn:n-left}. This completes the proof.
\end{proof}

Although Proposition \ref{prop:almost-univ} (2), (3) provide effective bounds for almost universality, it remains somewhat too large to guarantee that every small integer is representable as a sum of squares in $\mca_{d,m}^2$. In the following lemma, we derive a sharper bound, from which Lemma \ref{lem:rep-small} suffices to complete the verification of $\su$ for $d\equiv \pm 1\Mod{m}$.

\begin{lem}\label{lem:rep-1,m-effective}
	Let $m\ge6$ be an integer, $e\in\{1,2\}$ be such that $e\equiv m\Mod{2}$, and define $M=1$ if $m\equiv1\Mod{2}$, $M=4$ if $m\equiv 2\Mod{4}$, and $M=2$ if $m\equiv0\Mod{4}$. Then every integer $n\ge 2m(m-1)^2$ can be written as a sum of at most $Mm+5-e$ squares in $\mca_{d,m}^2$.
\end{lem}
\begin{proof}
    Write $n=mt+r_0$ with unique integers $t,r_0\in\z$ such that $0\le r_0 < m$. Let $r$ be the unique integer satisfying $m+6-e\le r\le Mm+5-e$ and $t\equiv \frac{r-r_0}{m}\Mod{M}$. Note that the equation
	\[
		n=\sum_{i=1}^r (mx_i+1)^2 
	\]
    has a solution $x_i\in\z$ if and only if
	\[
		\tilde{t}:=t-\frac{r-r_0}{m}=m\left(\sum_{i=1}^4 x_i^2\right) + 2\left(\sum_{i=1}^4 x_i\right) + \sum_{i=5}^r (mx_i^2+2x_i).
	\]
    
	Define the set
    \[
        S=\begin{cases}
            \{k(m-2) \mid 1\le k \le m\}\cup\{ k(m+2)\mid 0 \le k \le m-1\}& \text{if } m\equiv1\Mod{2},\\
            \{ k(\frac{m}{2}+1) \mid 0\le k \le m-1\}& \text{if } m\equiv0\Mod{2}.            
        \end{cases}
    \]
    One may observe that
	\begin{enumerate}
	\item [(1)] $S\subseteq \left\{\sum_{i=5}^{Mm+5-e} \frac{1}{e}(mx_i^2+2x_i) \mid x_i \in \{0,\pm1\}\right\}$,
	\item [(2)] there exists $s\in S$ such that $\frac{\tilde{t}}{e}\equiv \frac{2}{e}+s \Mod{\frac{m}{e}}$ and $\frac{\tilde{t}}{e}\not\equiv \frac{2}{e}+s\Mod{2\frac{m}{e}}$.
	\end{enumerate} 
	The second property follows since $S$ form a complete system of residues modulo $2\frac{m}{e}$. Choosing $s\in S$ satisfying (2), we set 
    \[
        a:=\frac{\tilde{t}/e-(2/e+s)}{m/e},
    \]
    which is an odd positive integer since 
    \[
        \tilde{t}-2-es=\frac{n-r}{m}-2-es\ge 2(m-1)^2-\frac{4m+3}{m}-2-(m-1)(m+2)>0.
    \]
    
    By Lemma \ref{lem:GenCauchy}, there exist $x_1,x_2,x_3,x_4\in\z$ such that
	\[
			x_1^2+x_2^2+x_3^2+x_4^2=a \quad \text{and} \quad x_1+x_2+x_3+x_4=1.
	\]
	Let $x_5,\ldots,x_r\in \{0,\pm1\}$ be such that $es=\sum_{i=5}^r (mx_i^2+2x_i)$. Then
	\[
		\sum_{i=1}^r (mx_i+1)^2 = m\left( m\sum_{i=1}^r x_i^2 + 2\sum_{i=1}^r x_i\right)+r = m(ma+2+es)+r = m\tilde{t}+r=n.
	\]
    This completes the proof.
\end{proof}

\begin{proof}[Proof of Theorem \ref{thm:su-main}]
    By Lemma \ref{lem:su-necessary}, we may assume that $d\equiv\pm1\Mod{m}$. For $m=2,3,4$, we have
    \[
        \su[\mca_{d,2}]=\su[\mca_{d,4}]=\su[\z\setminus2\z]=10 \quad \text{and} \quad \su[\mca_{d,3}]=\su[\z\setminus3\z]=6 \ 
    \]
    by \cite[Propositions 2.3 and 2.4]{KO}. 
    
    For $m\ge8$, note that $m^2-2m\ge 4m+3$. Hence by Lemmas \ref{lem:rep-small} and \ref{lem:rep-1,m-effective}, every positive integer can be written as a sum of at most $m^2-2m$ squares in $\mca_{d,m}^2$. Thus, we have $\su\le m^2-2m$.  On the other hand, By Lemma \ref{lem:su-necessary} and Theorem \ref{thm:asu-main}, $\su\ge \max(m^2-2m,\asu)=m^2-2m$. Therefore, $\su=m^2-2m$.

    If $m=5$ or $m=7$, then every integer $n\ge \frac{1}{5}m^4+m$ is a sum of at most $m+4$ squares in $\mca_{d,m}^2$ by Proposition \ref{prop:almost-univ}. For $m=7$, a direct check shows that every positive integer $n<\frac{1}{5}\cdot7^4+7=487.2$ is a sum of at most $m^2-2m=35$ squares in $\mca_{d,m}^2$, hence $\su[\mca_{d,7}]=m^2-2m$. On the other hand, for $m=5$, all positive integers $n\le \frac{1}{5}\cdot5^4+5=130$ are sums of at most $m^2-2m=15$ squares in $\mca_{d,m}^2$, except for $n=31$. Thus, $\su[\mca_{d,5}]=16$.

    Finally, for $m=6$, we show that $\su=26$. By Lemma \ref{lem:su-necessary} and Theorem \ref{thm:asu-main}, we have 
    \[
        \su[\mca_{d,6}]\ge \max(m^2-2m,4m+2)=26.
    \]
    On the other hand, every positive integer $n\ge 2m(m-1)^2=300$ is a sum of at most 27 squares in $\mca_{d,6}^2$ by Lemma \ref{lem:rep-1,m-effective}, and a direct check confirms this also holds for all $n<300$. Therefore, to prove that $\su[\mca_{d,6}]=26$, it suffices to show that every $n=mt+3$ with $t\equiv 0\Mod{4}$ can be written as a sum of three squares in $\mca_{d,6}^2$. 
    
    Let $n=24t'+3$ with $t'\in\n$. Since $n\equiv 3\Mod{8}$, it is representable as a sum of three squares of odd integers. We claim that there exists a representation $n=x^2+y^2+z^2$ with $x,y,z\in\z$ such that $(xyz,6)=1$. If $n\not \equiv0\Mod{9}$, then any such representation satisfies $(xyz,3)=1$. If $n\equiv0\Mod{9}$, then by \cite[Lemma 2.2]{KO}, there exists a solution $x,y,z\in\z$ with $(xyz,3)=1$. Since all $x,y,z$ are odd, it follows that  $(xyz,6)=1$. By changing signs if necessary, we may assume that $x,y,z\in\mca_{d,6}$, hence $n$ is a sum of three squares in $\mca_{d,6}^2$. This completes the proof.
\end{proof}

\end{document}